\documentclass[12pt, twoside, leqno]{article}

\usepackage{amsmath,amsthm}
\usepackage{amssymb}

\usepackage{enumitem}

\usepackage[T1]{fontenc}

\usepackage{mathtools}
\usepackage{amsmath}
\usepackage[%
    style=alphabetic,
    sorting=nty,
    sortcites=true,
    maxbibnames=99,
    doi=false,
    url=false,
    isbn=false,
    giveninits=true
    ]{biblatex}
\AtEveryBibitem{
  \ifentrytype{unpublished}{  
  \clearfield{note}
  }
}
\DeclareFieldFormat*{title}{\textit{#1}}
\DeclareFieldFormat{journaltitle}{#1\isdot}
\DeclareFieldFormat{pages}{#1}

\newtheorem{thm}{Theorem}[section]
\newtheorem{cor}[thm]{Corollary}
\newtheorem{lem}[thm]{Lemma}

\newtheorem{fact}[thm]{Fact}

\theoremstyle{definition}

\newtheorem{rem}[thm]{Remark}
\newtheorem{ex}[thm]{Example}

\numberwithin{equation}{section}

\newcommand{\R}{\mathbb{R}}
\newcommand{\Rd}{{\R^d}}
\newcommand{\EAm}{(m)}
\newcommand{\XSm}{(\mu)}
\newcommand{\LpE}{L^p\EAm}

\newcommand{\E}{\mathcal{E}}
\newcommand{\D}{\mathcal{D}}
\newcommand{\DpL}{\D(L_p)}
\newcommand{\DEp}{\D(\E_p)}
\newcommand{\C}{C}
\newcommand{\intE}{\int\limits_E}
\newcommand{\iintEE}{\iint\limits_{E\times E}}
\newcommand{\iintEEd}{\iint\limits_{\mathclap{E\times E\setminus\diag}}}
\newcommand{\dt}{\,dt}
\newcommand{\goodfunctions}{\mathcal{U}}
\newcommand{\diag}{\mathrm{diag}}
\newcommand{\indyk}{\mathbf{1}}

\DeclareMathOperator*\sgn{sgn}
\DeclareMathOperator*\supp{supp}

\begin{document}


\baselineskip=17pt


\title{Hardy--Stein identity for pure-jump Dirichlet forms}

\author{Micha\l\ Gutowski\\
Faculty of Pure and Applied Mathematics\\ 
Wroc\l{}aw University of Science and Technology\\
Wyb. Wyspia\'nskiego 27\\
50-370 Wroc\l{}aw, Poland\\
E-mail: michal.gutowski@pwr.edu.pl
}

\date{}

\maketitle


\renewcommand{\thefootnote}{}



\renewcommand{\thefootnote}{\arabic{footnote}}
\setcounter{footnote}{0}


\begin{abstract}
We prove the $L^p$ variant of the Hardy--Stein identity for Sobolev--Bregman forms associated with pure-jump Dirichlet forms, under rather mild assumptions. Along the way, we obtain a general result in terms of the $p$\nobreakdash-form defined in a more abstract way.
\end{abstract}

\section{Introduction}
\label{sec:introduction}

In this article we prove the Hardy--Stein identity for a rather general regular conservative pure-jump Dirichlet form
\begin{align}
\label{eq:E2=iintFp}
    \E(u,v)
    =
    \frac{1}{2}\iintEEd (u(y)-u(x))(v(y)-v(x))\,J(dx,dy)
\end{align}
on a locally compact separable metric space $E$, with Borel $\sigma$-algebra $\mathcal{B}$ and a positive Radon measure $m$ such that $\supp(m)=E$.
Here $\diag$ is the diagonal of the cartesian product $E\times E$ and $J$ is the symmetric jumping measure on $E\times E\setminus\diag$.
Let $(P_t)_{t\geq0}$ be a strongly continuous semigroup of contractions associated to the Dirichlet form $\E$.

The following is the main result of this paper.
\begin{thm}
\label{thm:HSp-j}
    Let $p\in(1,\infty)$.
    Let $\E$ be a regular conservative pure-jump Dirichlet form
    given by \eqref{eq:E2=iintFp}.
    Assume that:
    \begin{enumerate}[label=(\roman*)]
        \item \label{ass:continuous}
        For every $t>0$ and $f\in\LpE$ we have
        $P_tf\in\C(E)$, that is, $P_tf$ is continuous on $E$.
        
        \item \label{ass:diff}
        The semigroup $(P_t)_{t\geq0}$ is strongly stable on $\LpE$, i.e., for every $f\in\LpE$,
            \begin{align}
            \label{lim:|P_Tf|to0}
                \|P_T f \|_p \to 0
                \qquad \mbox{when }T\to\infty.
            \end{align}
    \end{enumerate}
    If $f\in\LpE$,
    then
    \begin{align}
    \label{eq:HSpure-jump}
        \intE |f|^p \,dm
        =
        \int\limits_0^{\infty} \quad\iintEEd
        F_p(P_tf(x),P_tf(y)) \,J(dx,dy)dt.
    \end{align}
    Here
    $
        F_p(a,b):=
        |b|^p-|a|^p - pa^{\langle p - 1 \rangle}(b-a)
    $
    is the Bregman divergence.
\end{thm}

In general, it is known that for every regular Dirichlet form $\E$, there exists (unique in a certain sense) $m$-symmetric Hunt process whose Dirichlet form is $\E$. For more details see \cite[Chapter 7]{MR2778606}. On the other hand, there is a one-to-one correspondence between the family of closed symmetric forms $\E$ on $L^2\EAm$ and the family of non-positive definite self-adjoint operators $L$ on $L^2\EAm$ (see \cite[Theorem 1.3.1.]{MR2778606}). We use these facts to present some examples of Dirichlet forms within the present settings.

\begin{ex}
\label{ex:alfa}
    The classical example of a pure-jump Dirichlet form that fulfills our assumptions is the Dirichlet form associated with the fractional Laplacian. More precisely, let $0<\alpha<2$ and denote by $L=-(-\Delta)^{\alpha/2}$ the fractional Laplacian on $E=\Rd$ equipped with the Lebesgue measure. This operator induces the Dirichlet form given by \eqref{eq:E2=iintFp}, where the jumping measure satisfies $J(dx,dy) = \nu(y-x)dxdy$ and
    \begin{align*}
        \nu(y) := \frac{\mathcal{A}_{d,\alpha}}{|y|^{d+\alpha}}.
    \end{align*}
    In the above formula the constant $\mathcal{A}_{d,\alpha}$ is given by
    \begin{align*}
        \mathcal{A}_{d,\alpha} := \frac{2^\alpha\Gamma\left(\frac{d+\alpha}{2}\right)}{\pi^{d/2}\Gamma\left(-\frac{\alpha}{2}\right)}.
    \end{align*}
    It is well-known that assumptions \ref{ass:continuous} and \ref{ass:diff} hold in this case.
\end{ex}

\begin{ex}
    More generally, Theorem~\ref{thm:HSp-j} applies to Dirichlet forms associated with pure-jump symmetric L\'{e}vy processes on $E=\Rd$ with L\'{e}vy measure $\nu$ fulfilling the Hartman--Wintner condition, i.e.,
    \begin{align*}
        \lim_{|\xi|\to\infty} \frac{\psi(\xi)}{\log(|\xi|)} = \infty,
    \end{align*}
    where $\psi$ is the characteristic exponent of L\'{e}vy process given by
    $
    \psi(\xi)=\int_\Rd \left(1-\cos(\xi\cdot x)\right) \,\nu(dx).
    $
    Then the Dirichlet form is given by \eqref{eq:E2=iintFp}, where the jumping measure satisfies $J(dx,dy) = \nu(dy-x)dx$.
    It is known that assumptions \ref{ass:continuous} and \ref{ass:diff} hold in this case. For more details, we refer the reader to \cite{MR3556449} and \cite{HS}, where the Hardy--Stein identity in this context was already considered.
\end{ex}

The following example show that the main result of this paper allows us to formulate the Hardy--Stein in more general cases. It concerns more general spaces than the Euclidean space $\Rd$.

\begin{ex}
\label{ex:d-set}
    Let $n\in\mathbb{N}_+$ and let $d$ be any real number from $(0,n]$. Let $E$ be a closed $d$\nobreakdash-set in $\R^n$ with measure $m$. That is, there are constants $C_1,C_2>0$ such that
    \begin{align}
    \label{neq:d-set}
        C_1 r^d
        \leq
        m(B(x,r))
        \leq
        C_2 r^d,
        \quad
        \mbox{for all } x\in E, r\in(0,1].
    \end{align}
    Here $B(x,r)$ is the intersection of the set $E$ and a ball in $\R^n$ of radius $r$ centered at $x$. One may show, that $m$ is equivalent to the $d$\nobreakdash-dimensional Hausdorff measure restricted to $E$. Typical examples of such sets are self-similar sets and Riemannian manifolds embedded into a Euclidean space $\R^n$.
    In addition to \eqref{neq:d-set}, we assume that
    \begin{align}
    \label{ass:dsets}
        m(B(x,r))
        \leq
        C_2 r^d,
        \quad
        \mbox{for every } x\in E, r>0.
    \end{align}
    For example, assumption \eqref{ass:dsets} is fulfilled for $E$ bounded in $\R^n$.
    
    Fix $0<\alpha<2$. Let $\phi:E\times E\to(0,\infty)$ be a symmetric function such that there are some constants $C_3,C_4>0$ such that
    \begin{align*}
        C_3 \leq \phi(x,y) \leq C_4
        \quad \mbox{for } m\mbox{-a.e. } x,y\in E.
    \end{align*}
    Consider the Dirichlet form \eqref{eq:E2=iintFp} with jumping measure
    \begin{align*}
        J(dx,dy) = \frac{\phi(x,y)}{|y-x|^{d+\alpha}} m(dx)m(dy).
    \end{align*}
    The process associated with the Dirichlet form within the present settings is called the stable-like process.
    
    For basic properties of such Dirichlet form, we refer the reader to \cite{MR2008600}.
    One may show that this Dirichlet form is conservative.
    One can show that assumption \ref{ass:continuous} holds. For the reader's convenience, we sketch the proof of this claim in Appendix \ref{apx:d-set}.
    Condition \ref{ass:diff} is satisfied whenever $m(E)=\infty$. If $m(E)<\infty$, then \ref{ass:diff} fails. In this case, however, modified Hardy--Stein identity holds; see Appendix \ref{apx:ss}.
\end{ex}

Along the way of proving the main result, in Theorem~\ref{thm:HS} we give a more general variant of the Hardy--Stein identity:
\begin{align*}
    \intE |f|^p \,dm
    =
    p\int\limits_0^{\infty} \E_p[P_tf] \dt,
    \quad
    f\in\LpE,
\end{align*}
where $\E_p$ is the $p$\nobreakdash-form defined as the limit of appropriate approximate forms $\E^{(t)}(u,u^{\langle p-1\rangle})=\langle u-P_tu,u^{\langle p-1\rangle} \rangle/t$. The $p$\nobreakdash-form may be treated as an extension of the classical quadratic Dirichlet form.

In the case of pure-jump Dirichlet forms, in Theorem~\ref{thm:Ep=iintFp} we derive the explicit formula for the $p$\nobreakdash-form
\begin{align*}
    \E_p[u] = \frac{1}{p} \iintEEd F_p(u(x),u(y)) \,J(dx,dy)
\end{align*}
for continuous functions $u$ from the domain $\DEp$ of $p$\nobreakdash-form.

The Hardy--Stein identity was originally proved for the classical Laplace operator as a consequence of Green's theorem and the chain rule $\Delta u^p = p(p-1)u^{p-2}|\nabla u|^2+pu^{p-1}\Delta u$. We refer to Lemma 1 and Lemma 2 in Stein \cite[p. 86--88]{MR0290095}. The identity in the non-local case can be shown using analytic methods as in \cite{MR3556449} or in \cite{HS}. On the other hand, in \cite{MR3994925} the identity is derived from It\^{o}'s lemma.

The Hardy--Stein identity has found applications in the Littlewood--Paley theory, especially to the proof of $L^p$-boundedness of the square function and Fourier multipliers. For more details, we refer to \cite{MR3556449} and \cite{MR3994925}. The Hardy--Stein identity also gives a characterization of Hardy spaces \cite{MR3251822} and it is used to prove Douglas-type identities \cite{bogdan2020nonlinear}.

The goal of this paper is to extend the previous Hardy--Stein type identities to more general Dirichlet forms which correspond to pure-jump regular Dirichlet form under certain mild assumptions. Our main tool is the theory of Dirichlet forms and $p$\nobreakdash-forms. To prove Theorem~\ref{thm:HS} we adopt the approach from \cite[Theorem 15]{HS}.

The $p$\nobreakdash-form (or Sobolev--Bregman form) corresponding to the fractional Laplacian (defined as a double integral of the Bregman divergence) was studied in \cite{lenczewska2021sharp}, \cite{bogdan2021optimal}, and for more general L\'{e}vy operators in \cite{HS}, but similar expression appeared already in \cite[Lemma 7.2]{MR1386760}. A similar object was frequently used also in \cite{bogdan2020nonlinear}.

In this work, we begin with a non-standard definition of the $p$\nobreakdash-form -- as a limit of appropriate approximating forms $\E^{(t)}(u,u^{\langle p-1\rangle})$. This gives access to an extended class of operators. The same definition was used in the case of the Brownian motion in \cite[Section 8]{HS}. At this moment it is known that this definition is equivalent to the commonly used definition in terms of the Bregman divergence for pure-jump L\'{e}vy operators \cite[Lemma 7]{bogdan2021optimal}, \cite[Proposition 13]{HS}. In Theorem~\ref{thm:Ep=iintFp} we show that for more general pure-jump Dirichlet form the equivalence remains true for continuous functions.
Equivalence of the two definitions for discontinuous functions in the domain of the $p$\nobreakdash-form remains an open problem.

Independently from the main subject of this article, the Hardy--Stein identity, we are interested in the relationship between the domains of the Dirichlet form and the $p$\nobreakdash-form. In \cite[Lemma 7]{bogdan2021optimal} and \cite[Proposition 13]{HS} it was shown that for a pure-jump L\'{e}vy operators, with certain assumptions about the L\'{e}vy measure,
the domain $\DEp$ of the $p$\nobreakdash-form  consists of functions of the form $u^{\langle p/2\rangle}$, where $u$ is a function in the domain $\D(\E)$ of the Dirichlet form.
In Theorem~\ref{thm:Ep=iintFp} we show the inclusion $\DEp\subseteq\D(\E)^{\langle p/2\rangle}$ in the general pure-jump case, and we conjecture that, in fact, equality holds.

The structure of the article is as follows. In Section \ref{sec:prelimitaries} we introduce the notions of a Dirichlet form, its semigroup, the definition of the corresponding $p$\nobreakdash-form, and derivatives of $L^p$-valued functions and we discuss their basic properties. The Hardy--Stein identity in the general case is proved in Section \ref{sec:HS}. In Section \ref{sec:pure-jump} we consider pure-jump Dirichlet forms and we show that the corresponding $p$\nobreakdash-form $\E_p$ for continuous functions is given by a double integral. This proves the Hardy--Stein identity for such Dirichlet forms (Theorem~\ref{thm:HSp-j}). In Appendix \ref{apx:ss} we discuss the sufficient condition for assumption \ref{ass:diff} and the situation when this assumption is not necessarily true.
In Appendix \ref{apx:d-set} we prove assumption \ref{ass:continuous} in case of Example~\ref{ex:d-set}.

\section{Preliminaries}
\label{sec:prelimitaries}

We consider a regular conservative Dirichlet form $\E$ on $L^2(E,\mathcal{B}, m)$ where $(E,\mathcal{B}, m)$ is a measure space.
For simplicity we will write $\LpE:=L^p(E,\mathcal{B}, m)$ for any $p\in[1,\infty]$.
We assume that $E$ is a locally compact separable metric space and $m$ is a positive Radon measure on $E$ such that $\supp(m)=E$, defined on the $\sigma$-algebra $\mathcal{B}$ of all Borel sets in $E$.
By $\C(E)$ we denote the class of continuous functions on $E$.
Let $(P_t)_{t\geq0}$ be the strongly continuous semigroup of contractions associated with the Dirichlet form $\E$.
Recall that for every $p\in[1,\infty)$ $(P_t)_{t\geq0}$ is a strongly continuous semigroup of contractions on $\LpE$.
For a function $u\in\LpE$ we have
\begin{align*}
    P_tu(x) = \intE u(y)\,P_t(x,dy),
    \qquad t\geq0, x\in E.
\end{align*}
Here $P_t(x,dy)$ is the probability kernel associated with the operator $P_t$.
To emphasize symmetry, we write $P_t(dx,dy):=P_t(x,dy)m(dx)$; then $P_t(dx,dy)=P_t(dy,dx)$.

We use the notation
\begin{align*}
	a^{\langle \kappa \rangle} := \left|a \right|^\kappa \sgn a,
\end{align*}
whenever above expression makes sense.
Note that
\begin{align*}
	(
    |x|^\kappa
    )' = \kappa x^{\langle \kappa-1 \rangle}
    ,\quad
    \mbox{if }
    x\in\R,\ \kappa>1
    \mbox{ or }
     x\in\R\!\setminus\!\{0\},
\end{align*}
and
\begin{align*}
	(
    x^{\langle \kappa \rangle}
    )' = \kappa |x|^{\kappa-1 }
     ,\quad
     \mbox{if }
     x\in\R,\ \kappa\geq1
    \mbox{ or }
     x\in\R\!\setminus\!\{0\}.
\end{align*}

Let $p,q\in(1,\infty)$ with $p^{-1}+q^{-1}=1$. For $u\in\LpE$, $v\in L^q\EAm$ we use the notation
\begin{align*}
    \langle u,v \rangle := \intE u(x) v(x) \,m(dx).
\end{align*}
For $t>0$ and $u\in\LpE$, $v\in L^q\EAm$ we define
\begin{align}
    \E^{(t)}(u,v)
    :=
    \frac{1}{t}
    \langle u-P_tu,v \rangle.
\end{align}

Let $u\in\LpE$. We define the nonlinear functional
\begin{align*}
    \E_p[u] := \lim_{t\to0^+} \E^{(t)}\left(u,u^{\langle p-1\rangle}\right)
\end{align*}
with its natural domain
\begin{align*}
    \DEp := \left\{u\in\LpE\!: \mbox{finite } \lim_{t\to0^+} \E^{(t)}\left(u,u^{\langle p-1\rangle}\right) \mbox{ exists}\right\}.
\end{align*}
We call $\E_p$ the \emph{$p$\nobreakdash-form} corresponding to the Dirichlet form $\E$.

When $p=2$, then $\E_2$ is just the usual Dirichlet form $\E(u,u)$ with domain $\D(\E_2)=\D(\E)$.

It is well-known that if $u\in L^2\EAm$, then $\E^{(t)}(u,u)$ is non-increasing as a function of $t$ \cite[Lemma 1.3.4]{MR2778606}.

We consider the infinitesimal generator $L_p$ of the semigroup $(P_t)_{t\geq0}$ on $\LpE$:
\begin{align}
\label{eq:Lp-generator}
    L_p u := \lim_{t\to0^+} \frac{1}{t}(P_tu - u)
    \qquad
    \mbox{in }
    \LpE
\end{align}
with the natural domain
\begin{align*}
    \DpL := \left\{u\in\LpE\!: \lim_{t\to0^+} \frac{1}{t}(P_tu - u) \mbox{ exists in }\LpE\right\}.
\end{align*}
Of course, $\DpL\subseteq\DEp$ and
\begin{align*}
    \E_p[u] = -\langle L_p u, u^{\langle p-1\rangle} \rangle,
    \qquad
    u\in\DpL.
\end{align*}

We use the \emph{Bregman divergence}: a function $F_p:\R\times\R\to\R$, where $p>1$, defined by
\begin{align*}
    F_p(a,b)
    :=
    |b|^p-|a|^p - pa^{\langle p - 1 \rangle}(b-a).
\end{align*}
We also use \emph{symmetrized Bregman divergence}
\begin{align*}
    H_p(a,b)
    :=
    \frac{1}{2}
    \left(
    F_p(a,b)+F_p(b,a)
    \right)
    =
    \frac{p}{2}
    (b-a)
    \left(
    b^{\langle p - 1 \rangle} - a^{\langle p - 1 \rangle}
    \right).
\end{align*}
Note that $F_p(a,b)$ is the second-order Taylor remainder of the convex map $\R\ni a\mapsto |a|^p\in\R$, we have $F_p\ge 0$ and also $H_p\ge 0$. Furthermore, $F_2(a,b)=H_2(a,b)=(b-a)^2$.

The following estimate was proved in \cite[Lemma 1]{MR1160303} for $H_p$ in place of $F_p$.
\begin{lem}
\label{lem:FpsimD}
    Let $p\in(1,\infty)$. There exist constants $c_p,C_p>0$ such that
    \begin{align}
    \label{neq:FpsimD}
        c_p (b^{\langle p/2\rangle} - a^{\langle p/2\rangle})^2
        \leq
        F_p(a,b)
        \leq
        C_p (b^{\langle p/2\rangle} - a^{\langle p/2\rangle})^2
    \end{align}
    for all $a,b\in\R$.
\end{lem}
\begin{proof}
    If $a=0$, we have $F_p(a,b)=|b|^p$ and the statement is obvious. If $a\neq0$, then we let  $x:=b/a$ and we arrive
    at the following equivalent formulation of \eqref{neq:FpsimD}:
    \begin{align}
    \label{neq:FpsimDx}
        c_p (x^{\langle p/2\rangle} - 1)^2
        \leq
        F_p(x,1)
        \leq
        C_p (x^{\langle p/2\rangle} - 1)^2,
    \end{align}
    where $F_p(x,1) = |x|^p - 1 - p(x-1)$. The above expressions define continuous and positive functions of $x\neq1$. Therefore to prove \eqref{neq:FpsimDx} it is enough to notice that
    \begin{align*}
        \lim_{x\to\pm\infty}
        \frac{F_p(x,1)}{\left(x^{\langle p/2\rangle} - 1\right)^2}
        =
        1
    \end{align*}
    and, by L'H\^{o}pital's rule,
    \begin{align*}
        \lim_{x\to1}
        \frac{F_p(x,1)}{\left(x^{\langle p/2\rangle} - 1\right)^2}
        =
        \lim_{x\to1}
        \frac{x^{p-1}-1}{\left(x^{p/2} - 1\right)x^{(p-2)/2}}
        =
        \frac{2(p-1)}{p}.
    \end{align*}
    The above limits are finite and positive. This completes the proof.
\end{proof}

For every $u\in L^1\EAm$, by symmetry of $P_t$, we have
\begin{align}
\label{eq:intPtu=intu}
    \intE P_tu(x) \,m(dx)
    &=
    \intE\left(\intE u(y) \,P_t(x,dy)\right)\,m(dx)
    \\ \nonumber
    &=
    \intE u(y) \left(\intE \,P_t(y,dx)\right)\,m(dy)
    =
    \intE u(x) \,m(dx).
\end{align}
Let $u\in\LpE$. Using \eqref{eq:intPtu=intu} for $|u|^p\in L^1\EAm$ we can write
\begin{align}
\label{eq:Et=intintFp}
    \E^{(t)}\left(u,u^{\langle p-1\rangle}\right)
    &=
    \frac{1}{t} \langle u-P_tu,u^{\langle p-1\rangle} \rangle
    \\
    &= \nonumber
    -\frac{1}{t} \iintEE u^{\langle p-1\rangle}(x) (u(y)-u(x)) \,P_t(dx,dy)
    \\
    &= \nonumber
    \frac{1}{pt} \intE P_t(|u|^p)(x) \,m(dx)
    -
    \frac{1}{pt} \intE |u|^p(x) \,m(dx)
    \\
    &\quad- \nonumber
    \frac{1}{t} \iintEE u^{\langle p-1\rangle}(x) (u(y)-u(x)) \,P_t(dx,dy)
    \\
    &= \nonumber
    \frac{1}{pt} \iintEE F_p(u(x),u(y)) \,P_t(dx,dy).
\end{align}
In particular, we see that $\E^{(t)}\left(u,u^{\langle p-1\rangle}\right)\geq0$, and so $\E_p[u]\geq0$ whenever $u\in\DEp$.

By symmetry of $P_t$ we can also write
\begin{align}
\label{eq:Et=intintHp}
    \E^{(t)}\left(u,u^{\langle p-1\rangle}\right)
    =
    \frac{1}{pt} \iintEE H_p(u(x),u(y)) \,P_t(dx,dy).
\end{align}

Let $p\in[1,\infty)$ and let $I\subseteq[0,\infty)$ be an interval. For a mapping
$I \ni t \mapsto u(t) \in \LpE$ we denote
\begin{align*}
    \Delta_hu(t) := u(t+h)-u(t) \quad \mbox{if }t,t+h\in I.
\end{align*}

We say that $u$ is \emph{continuous} on $I$ with values in $\LpE$ if $\Delta_hu(t) \to 0$ in $\LpE$ as $h\to 0$ for every $t\in I$,
and we say that $u$ is \emph{continuously differentiable} (or shortly $\C^1$) on $I$ with values in $\LpE$ if $u'(t) := \lim_{h \to 0} \frac{1}{h} \Delta_hu(t)$ exists in $\LpE$ for every $t\in I$ and the mapping $I \ni t \mapsto u'(t) \in \LpE$ is
continuous.

The following two elementary results are proved rigorously in \cite[Lemmas 15, 16]{MR4372148}.

\begin{lem}[\cite{MR4372148}]
    Let $p\in(1,\infty)$. If $I\ni t \mapsto u(t)$ is $\C^1$ on $I$ with values in $\LpE$, then $|u|^p$ is $\C^1$ on $I$ with values in $L^1\EAm$ and $(|u|^p)' = p u^{\langle p - 1 \rangle} u'$.
\end{lem}

Let $f\in\LpE$ and let $u(t) := P_tf\in\LpE$.
If $f\in\DpL$, then $u'(t) = L_p P_t f = P_t L_p f=L_p u(t)$.
We know that if $p>1$, then $(P_t)_{t\geq0}$ is an analytic semigroup on $\LpE$ \cite[p. 67]{MR0252961}. In particular, for every $t>0$ and $f\in\LpE$ the derivative
$\frac{d}{dt} P_t f = u'(t)$
exists in $\LpE$. Hence $P_t f \in \DpL$ and
$u'(t) = L_p P_t f = L_p u(t)$.

\begin{cor}[\cite{MR4372148}]
\label{cor:ppp}
    Let $f\in\DpL$ and $u(t) := P_tf$.
    Then $|u(t)|^p$ is $\C^1$ on $[0,\infty)$ with values in $L^1\EAm$, with derivative
    \begin{equation}
    \label{eq:ppp}
        (|u(t)|^p)' =p u(t)^{\langle p - 1 \rangle} u'(t)= p u(t)^{\langle p - 1 \rangle} L_p u(t), \quad t \geq 0.
    \end{equation}
\end{cor}

\section{Hardy--Stein identity in the general case}
\label{sec:HS}

In this section we prove the Hardy--Stein identity for arbitrary regular Dirichlet form. The explicit form of the right-hand side of following identity depends on the specific Dirichlet form. In particular, in Section \ref{sec:pure-jump} we will give an explicit expression for pure-jump Dirichlet forms.

\begin{thm}
\label{thm:HS}
    Let $p\in(1,\infty)$.
    Assume that condition (ii) from Theorem~\ref{thm:HSp-j} holds.
    For every $f\in\LpE$,
    \begin{align}
    \label{eq:HS}
        \intE |f|^p \,dm
        =
        p\int\limits_0^{\infty} \E_p[P_tf] \dt.
    \end{align}
\end{thm}
\begin{proof}
    Consider first $f\in\DpL$ and fix $T>0$. Let $u(t):=P_tf$. By Corollary~\ref{cor:ppp}, $|u|^p$ is $\C^1$ on $[0,T]$ with values in $L^1\EAm$ with derivative $(|u(t)|^p)'=p u(t)^{\langle p-1\rangle} L_pu(t)$.
    The integral is a continuous linear functional on $L^1\EAm$, hence $[0,T]\ni t\mapsto \int_E |u(t)|^p \,dm$ is $\C^1$ and
    \begin{align*}
        \frac{d}{dt} \intE |u(t)|^p \,dm
        &=
        \intE \frac{d}{dt} |u(t)|^p \,dm
        =
        \intE pu(t)^{\langle p-1\rangle} L_pu(t) \,dm
        \\
        &=
        p\langle L_p u(t), u(t)^{\langle p-1\rangle} \rangle
        =
        -p \E_p[u(t)].
    \end{align*}
    Therefore, we can write
    \begin{align*}
        \intE |f|^p \,dm
        -
        \intE |P_Tf|^p \,dm
        &=
        -\left(
        \intE |u(T)|^p \,dm
        -
        \intE |u(0)|^p \,dm
        \right)
        \\
        &=
        -\int\limits_0^T \frac{d}{dt} \intE |u(t)|^p \,dm\dt
        \\
        &=
        p\int\limits_0^T \E_p[u(t)] \dt.
    \end{align*}
    
    From the strong stability of the semigroup (assumption \ref{ass:diff}), we have $\int_E |P_Tf|^p \,dm\to0$ when $T\to\infty$.
    Since $\E_p[u(t)]\geq0$, the right-hand side tends to $p\int_0^{\infty} \E_p[u(t)] \dt$ as $T\to\infty$. Therefore
    \begin{align*}
        \intE |f|^p \,dm = p\int\limits_0^{\infty} \E_p[P_tf] \dt.
    \end{align*}
    Next, we relax the assumption that $f\in\DpL$. Let $f$ be an arbitrary function in $\LpE$ and let $s>0$. Recall that $P_s f\in\DpL$. Thus, \eqref{eq:HS} holds
    for $P_sf$:
    \begin{align*}
        \intE |P_sf|^p \,dm = p\int\limits_s^{\infty} \E_p[P_tf] \dt.
    \end{align*}
    Since $(P_t)_{t\geq0}$ is a strongly continuous semigroup on $\LpE$ and $f \mapsto \int_E |f|^p\,dm$ is a continuous functional on $\LpE$, the left-hand side tends to $\int_E |f|^p \,dm$ when $s\to0^+$.
    Since $\E_p[P_tf]\geq0$,
    the right-hand side tends to $p\int_0^{\infty} \E_p[u(t)] \dt$ by the monotone convergence theorem.
\end{proof}

\begin{rem}
    Notice that without assumption \ref{ass:diff}, the identity takes form
    \begin{align*}
        \intE |f|^p \,dm
        -
        \lim_{T\to\infty} \intE |P_Tf|^p \,dm
        =
        p\int\limits_0^{\infty} \E_p[P_tf] \dt,
    \end{align*}
    and in particular, the limit on the left-hand side exists.
    This situation is discussed in more detail in Appendix \ref{apx:ss}.
\end{rem}

\section{Pure-jump Dirichlet forms}
\label{sec:pure-jump}

In this section we consider a pure-jump regular Dirichlet form: we assume that the Dirichlet form $\E$ is given by \eqref{eq:E2=iintFp}. Here and below, $\diag:=\{(x,y)\in E\times E\!: x=y\}$. Measure $J$, the so-called \emph{jumping measure}, is a symmetric positive Radon measure on the $E\times E\setminus\diag$.

Our goal is to propose explicit form of $p$\nobreakdash-form for such Dirichlet form.

\begin{lem}
    We have
    \begin{align}
    \label{lim:P_tvaguely}
        \frac{1}{t}P_t(dx,dy) \to J(dx,dy)
        \quad
        \mbox{vaguely on } E\times E\setminus\diag  \mbox{ when }t\to0^+.
    \end{align}
\end{lem}
An analogous result for the resolvent
rather than the semigroup $(P_t)$
was showed in \cite[(3.2.7)]{MR2778606}. The proof is similar and we omit it.

We consider the class $\goodfunctions$ of non-negative functions $f$ on $E\times E$ such that
\begin{align*}
    \lim_{t\to0^+} \frac{1}{t} \iintEE f(x,y) \,P_t(dx,dy)
    =
    \iintEEd f(x,y) \,J(dx,dy) < \infty.
\end{align*}

We know that if $u\in\D(\E)$ and $f(x,y)=(u(y)-u(x))^2$, then $f\in\goodfunctions$ because
\begin{align*}
    \lim_{t\to0^+} \frac{1}{t} \iintEE (u(y)-u(x))^2 \,P_t(dx,dy)
    &=
    \lim_{t\to0^+} 2\E^{(t)}(u,u)
    =
    2\E[u]
    \\
    &=
    \iintEEd (u(y)-u(x))^2\,J(dx,dy).
\end{align*}
Here we used \eqref{eq:Et=intintHp}.

\begin{lem}
\label{lem:if_g_good_then_also_f}
    Suppose that $0\leq f\leq g$, $f=g=0$ on $\diag$, $f,g\in\C(E\times E)$ and $g\in\goodfunctions$. Then $f\in\goodfunctions$.
\end{lem}
\begin{proof}
    Fix $\varepsilon>0$.
    Since $g\in\goodfunctions$, we have
    \begin{align*}
        \iintEEd g(x,y) \,J(dx,dy) < \infty.
    \end{align*}
    Therefore, 
    \begin{align*}
        \iintEEd f(x,y) \,J(dx,dy)
        \leq
        \iintEEd g(x,y) \,J(dx,dy)
        <
        \infty
    \end{align*}
    and there is a compact subset $K\subseteq E\times E\setminus\diag$ such that
    \begin{align}
    \label{neq:Kclessepsilon}
        \iint\limits_{K^c} g(x,y) \,J(dx,dy) < \varepsilon.
    \end{align}
    Let $\varphi\in\C_c(E\times E\setminus\diag)$ be such that $0\leq\varphi\leq1$ and $\varphi=1$ on $K$. Since $f$ is continuous, we have $\varphi\cdot f\in\C_c(E\times E\setminus\diag)$, hence from \eqref{lim:P_tvaguely} we obtain
    \begin{align}
    \label{lim:iintfvarphi}
        \lim_{t\to0^+}
        \frac{1}{t}
        \iintEEd
        \varphi(x,y) f(x,y)\,P_t(dx,dy)
        =
        \iintEEd
        \varphi(x,y) f(x,y)\,J(dx,dy).
    \end{align}
    Using \eqref{neq:Kclessepsilon} we can write
    \begin{align}
    \label{neq:int(1-varphi)fJ}
        \iintEEd
        (1-\varphi(x,y))f(x,y) \,J(dx,dy)
        &\leq
        \iintEEd
        (1-\varphi(x,y))g(x,y) \,J(dx,dy)
        \nonumber
        \\
        &\leq
        \iint\limits_{K^c}
        g(x,y) \,J(dx,dy)
        <
        \varepsilon.
    \end{align}
    We also have
    \begin{align*}
        \frac{1}{t}
        \iintEEd
        (1-\varphi(x,y))f(x,y) \,P_t(dx,dy)
        &\leq
        \frac{1}{t}
        \iintEEd
        (1-\varphi(x,y))g(x,y) \,P_t(dx,dy)
        \\
        &=
        \frac{1}{t}
        \iintEEd
        g(x,y) \,P_t(dx,dy)
        \\
        &\quad-
        \frac{1}{t}
        \iintEEd
        \varphi(x,y) g(x,y) \,P_t(dx,dy).
    \end{align*}
    Since $g\in\goodfunctions$ and $g$ is continuous, we have $\varphi\cdot g\in\C_c(E\times E\setminus\diag)$, and thus, using \eqref{lim:P_tvaguely}, we find that as $t\to0^+$, the right-hand side converges to
    \begin{align*}
        \iintEEd
        &g(x,y) \,J(dx,dy)
        -
        \iintEEd
        \varphi(x,y) g(x,y) \,J(dx,dy)
        \\
        &=
        \iintEEd
        (1-\varphi(x,y))g(x,y) \,J(dx,dy)
        \leq
        \iint\limits_{K^c}
        g(x,y) \,J(dx,dy)
        <
        \varepsilon.
    \end{align*}
    Here we used \eqref{neq:Kclessepsilon}. Therefore,
    \begin{align}
    \label{neq:int(1-varphi)fPt}
        \limsup_{t\to0^+}
        \frac{1}{t}
        \iintEEd
        (1-\varphi(x,y))f(x,y) \,P_t(dx,dy)
        <
        \varepsilon.
    \end{align}
    Finally, from \eqref{lim:iintfvarphi}, \eqref{neq:int(1-varphi)fJ} and \eqref{neq:int(1-varphi)fPt} we get
    \begin{align*}
        &\limsup_{t\to0^+}
        \left|\ 
        \frac{1}{t}\iintEEd f(x,y) \,P_t(dx,dy)
        -
        \iintEEd f(x,y) \,J(dx,dy)
        \right|
        \\
        &\leq
        \limsup_{t\to0^+}
        \left|\ 
        \frac{1}{t}\iintEEd \varphi(x,y) f(x,y) \,P_t(dx,dy)
        -
        \iintEEd \varphi(x,y) f(x,y) \,J(dx,dy)
        \right|
        \\
        &\quad+
        \limsup_{t\to0^+}
        \frac{1}{t}
        \iintEEd
        (1-\varphi(x,y))f(x,y) \,P_t(dx,dy)
        \\
        &\quad+
        \iintEEd
        (1-\varphi(x,y))f(x,y) \,J(dx,dy)
        <
        0+\varepsilon+\varepsilon = 2\varepsilon.
    \end{align*}
    Since $\varepsilon>0$ is arbitrary,
    \begin{align*}
        \lim_{t\to0^+}\frac{1}{t}\iintEEd f(x,y) \,P_t(dx,dy)
        =
        \iintEEd f(x,y) \,J(dx,dy).
    \end{align*}
\end{proof}

\begin{thm}
\label{thm:Ep=iintFp}
    Let $u\in\DEp$. Then $u^{\langle p/2\rangle}\in\D(\E)$. Moreover, if $u\in\C(E)$, then
    \begin{align}
    \label{eq:Ep=iintFp}
        \E_p[u] = \frac{1}{p} \iintEEd F_p(u(x),u(y)) \,J(dx,dy).
    \end{align}
\end{thm}
\begin{proof}
    Using \eqref{eq:Et=intintFp}, the fact that $F_2(a,b)=(b-a)^2$ and  Lemma~\ref{lem:FpsimD}, we find that
    \begin{align*}
        \E^{(t)}\left(u^{\langle p/2\rangle},u^{\langle p/2\rangle}\right)
        &=
        \frac{1}{2t}
        \iintEE
        \left(
        (u(y))^{\langle p/2\rangle} - (u(x))^{\langle p/2\rangle}
        \right)^2
        \,P_t(dx,dy)
        \\
        &\leq
        c_p^{-1} \frac{1}{2t}
        \iintEE
        F_p(u(x),u(y))
        \,P_t(dx,dy)
        \\
        &=
        \frac{p}{2c_p} \E^{(t)}\left(u,u^{\langle p-1\rangle}\right).
    \end{align*}
    Since $u\in\DEp$, the right-hand side converges to a finite limit $\frac{p}{2c_p} \E_p[u]$ as $t\to0^+$, and since the left-hand side is non-increasing as a function of $t$, a finite limit $\lim_{t\to0^+} \E^{(t)}\left(u^{\langle p/2\rangle},u^{\langle p/2\rangle}\right)$ exists, i.e., $u^{\langle p/2\rangle}\in\D(\E)$.
    
    Let us additionally assume that $u\in\C(E)$.
    Denote
    $f(x,y):=F_p(u(x),u(y))$,
    and 
    $g(x,y):=C_p\left(
    (u(y))^{\langle p/2\rangle} - (u(x))^{\langle p/2\rangle}
    \right)^2$, where $C_p$ is as in \eqref{neq:FpsimD}.
    Since $u^{\langle p/2\rangle}\in\D(\E)$, we have $g\in\goodfunctions$, and by \eqref{neq:FpsimD} we have $f\leq g$. Moreover, $f=g=0$ on $\diag$ and $f,g\in\C(E\times E)$ because $u\in\C(E)$. Therefore, we can use Lemma~\ref{lem:if_g_good_then_also_f} and conclude that $f\in\goodfunctions$. This means that
    \begin{align*}
        \E_p[u]
        &=
        \lim_{t\to0^+} \E^{(t)}\left(u,u^{\langle p-1\rangle}\right)
        =
        \lim_{t\to0^+} \frac{1}{pt}
        \iintEE
        F_p(u(x),u(y)) \,P_t(dx,dy)
        \\
        &=
        \frac{1}{p}
        \iintEEd
        F_p(u(x),u(y)) \,J(dx,dy).
    \end{align*}
\end{proof}

Now, we are ready to prove the main result of the article.
\begin{proof}[Proof of Theorem~\ref{thm:HSp-j}]
    It is enough to note that, by assumption, for every $t>0$ we have $P_tf\in\C(E)\cap\DEp$, and so we can rewrite formula \eqref{eq:HS} using \eqref{eq:Ep=iintFp} with $u=P_tf$.
\end{proof}

\appendix
\section{Strong stability of the heat semigroup}
\label{apx:ss}

In this section we discuss a sufficient condition for assumption \ref{ass:diff} as well as the question of how the main result of this paper changes without this assumption.

In general, when assumption \ref{ass:diff} is not necessarily true, the Hardy-Stein identity takes the following form
\begin{align}
\label{eq:HSwithoutSS}
    \intE |f|^p \,dm - \lim_{T\to\infty} \|P_T f\|_p^p
    =
    \int\limits_0^{\infty} \quad\iintEEd
    F_p(P_tf(x),P_tf(y)) \,J(dx,dy)dt.
\end{align}
The question is, under what assumptions the term $\lim_{T\to\infty} \|P_T f\|_p^p$ is equal to zero or can be written in a simpler form.

For $p\in[1,\infty]$ we define the operator $P_\infty\!:\LpE\to \LpE$ by
\begin{align*}
    P_\infty f := \lim_{T\to\infty} P_T f,
    \qquad
    \mbox{in }
    \LpE.
\end{align*}
First of all, the following fact holds.
\begin{fact}
\label{fact:ortopro}
    For $p=2$, the operator $P_\infty$ is the orthogonal projection onto the kernel of the generator $L_2$ (given by \eqref{eq:Lp-generator}).
\end{fact}
\begin{proof}
    Using spectral theorem (in the multiplication version) for the semigrup $(P_t)_{t\geq0}$ on $L^2\EAm$ there is a measure space $(X,\Sigma,\mu)$,
    function $\lambda\in L^\infty\XSm$
    (here, for any $p\in[1,\infty]$ we denote $L^p\XSm:=L^p(X,\Sigma,\mu)$)
    and a unitary operator
    $U\!:L^2\EAm\to L^2\XSm$ such that
    \begin{align}
    \label{eq:spec-Pt}
        U^* P_t U =  M_t 
    \end{align}
    where, for every $t\geq0$
    \begin{align*}
        M_tg(x) := e^{-t\lambda(x)}g(x),
        \qquad
        g\in L^2\XSm
    \end{align*}
    is the multiplication operator. Moreover, the following equality holds
    \begin{align}
    \label{eq:spec-L2}
        U^* (-L_2) U =  M ,
    \end{align}
    where
    $Mg(x) := \lambda(x)g(x)$.
    The operator $-L_2$ is a non-negative definite self-adjoint operator on $L^2\EAm$, therefore $\lambda\geq0$ $\mu$-almost everywhere.
    
    It is enough to note that for any $g\in L^2\XSm$
    \begin{align*}
         M_\infty g(x) := \lim_{T\to\infty} M_T g(x) = \indyk_{\{ \lambda=0\}}(x)g(x)
         \qquad
         \mbox{for $\mu$-almost every }x,
    \end{align*}
    i.e., $M_\infty$ is the orthogonal projection onto the kernel of the operator $M$.
    By \eqref{eq:spec-Pt} and \eqref{eq:spec-L2},
    $P_\infty=UM_\infty U^*$
    is the orthogonal projection onto the kernel of
    $-L_2 = UMU^*$.
\end{proof}

Using the above fact, we can show that when the jumping measure $J$ is irreducible in an appropriate sense, then for every $f\in\LpE$ the function $P_\infty f$ is constant a.e.

\begin{fact}
\label{fact:Pinfty=bar}
    Let $p\in(1,\infty)$. Assume that for every $A\in\mathcal{B}$
    \begin{align}
    \label{ass:accessible}
        m(A)>0, m(A^c)>0
        \quad \mbox{implies} \quad
        J(A\times A^c)>0.
    \end{align}
    Then for every $f\in\LpE$
    \begin{align*}
        P_\infty f = \Bar{f}
        \qquad
        m\mbox{-a.e.},
    \end{align*}
    where
    \begin{enumerate}[label=(\alph*)]
        \item $\Bar{f}\equiv0$ when $m(E)=\infty$;
        \item $\Bar{f}$ is equal to the mean value of the function $f$, i.e.,
            \begin{align*}
                \Bar{f} \equiv \frac{1}{m(E)}\intE f \,dm
            \end{align*}
            when $m(E)<\infty$.
    \end{enumerate}
\end{fact}
\begin{rem}
    One can prove similar results when the space $E$ can be divided into mutually non-accessible sets.
\end{rem}
\begin{proof}[Proof of the fact]
    Let us first assume that $p=2$. By Fact~\ref{fact:ortopro} the following equality holds
    \begin{align*}
        \E(P_\infty f,P_\infty f) = -\langle L_2P_\infty f, P_\infty f\rangle = 0.
    \end{align*}
    Therefore, by \eqref{eq:E2=iintFp} and assumption \eqref{ass:accessible} we obtain that $P_\infty f$ is constant $m$-a.e. Let $c$ be this constant.

    When $m(E)=\infty$, every $m$-a.e. constant function on $L^2\EAm$ must be equal to zero $m$-a.e. Otherwise, when $m(E)<\infty$, we have $f,P_\infty f\in L^2\EAm\subseteq L^1\EAm$ and we can write
    \begin{align*}
        \intE f \,dm
        &=
        \lim_{T\to\infty} \intE f \,dm
        =
        \lim_{T\to\infty} \intE P_Tf \,dm
        =
        \lim_{T\to\infty} \langle P_Tf, 1 \rangle
        =
        \langle P_\infty f, 1 \rangle
        \\
        &=
        \intE P_\infty f \,dm
        =
        c m(E)
    \end{align*}
    and therefore $c=\Bar{f}$. Here we used \eqref{eq:intPtu=intu}.

    Now, we relax the assumption about $p$. Let $p\neq2$ and fix $\varepsilon>0$. By density, there is a function $g\in L^1\EAm \cap L^\infty\EAm$ such that
    \begin{align}
    \label{neq:density}
        \|f-g\|_p < \varepsilon.
    \end{align}
    Of course, by the contraction property, we have also
    \begin{align}
    \label{neq:densityP_T}
        \|P_Tf-P_Tg\|_p < \varepsilon.
    \end{align}
    
    By an interpolation argument we have $g\in L^2\EAm\cap\LpE$ and therefore $P_\infty g=\Bar{g}$. Moreover, by conservativeness, we have
    $P_t\Bar{g} = \Bar{g}$.

    For $p>2$, using log-convexity of the $L^p$-norm with respect to $p$ and the contraction property, we get
    \begin{align*}
        \|P_Tg - \Bar{g}\|_p^p
        \leq
        \|P_Tg - \Bar{g}\|_2^2 \|P_Tg - \Bar{g}\|_\infty^{p-2}
        \leq
        \|P_Tg - \Bar{g}\|_2^2 \|g - \Bar{g}\|_\infty^{p-2}
        \to
        0
    \end{align*}
    when $T\to\infty$. Similarly, for $p<2$
    \begin{align*}
        \|P_Tg - \Bar{g}\|_p^p
        \leq
        \|P_Tg - \Bar{g}\|_2^{2p-2} \|P_Tg - \Bar{g}\|_1^{2-p}
        \leq
        \|P_Tg - \Bar{g}\|_2^{2p-2} \|g - \Bar{g}\|_1^{2-p}
        \to
        0
    \end{align*}
    when $T\to\infty$.

    We claim that in the case $m(E)<\infty$
    \begin{align}
    \label{neq:bars}
        \|\Bar{g}-\Bar{f}\|_p < m(E)^{1/p-1} \varepsilon.
    \end{align}
    Indeed, by Jensen's inequality
    \begin{align*}
        |\Bar{g}-\Bar{f}|^p
        =
        m(E)^{-p} \left| \intE (g-f) \,dm \right|^p
        \leq
        m(E)^{-p} \intE |g-f|^p \,dm
        <  m(E)^{-p} \varepsilon^p.
    \end{align*}
    This implies \eqref{neq:bars}.
    When $m(E)=\infty$, then, by definition $\Bar{f},\Bar{g}\equiv0$ hence of course $\|\Bar{g}-\Bar{f}\|_p=0$.

    Finally, in the case $m(E)<\infty$, using \eqref{neq:density}, \eqref{neq:densityP_T} and \eqref{neq:bars} we obtain
    \begin{align*}
        \|P_Tf - \Bar{f}\|_p
        &\leq
        \|P_Tf - P_Tg\|_p + \|P_Tg - \Bar{g}\|_p + \|\Bar{g}-\Bar{f}\|_p
        \\
        &<
        \varepsilon (1+m(E)^{1/p-1}) + \|P_Tg - \Bar{g}\|_p
    \end{align*}
    and therefore
    \begin{align*}
        \limsup_{T\to\infty} \|P_Tf - \Bar{f}\|_p
        \leq
        \varepsilon (1+m(E)^{1/p-1}).
    \end{align*}
    Since $\varepsilon>0$ was arbitrary,
    $\lim_{T\to\infty} \|P_Tf - \Bar{f}\|_p = 0$.

    Similarly
    $\lim_{T\to\infty} \|P_Tf - \Bar{f}\|_p =\lim_{T\to\infty} \|P_Tf\|_p = 0$
    in the case $m(E)=\infty$.
\end{proof}

\begin{cor}
    Let $p\in(1,\infty)$. Assume that \eqref{ass:accessible} holds. If $m(E)=\infty$, then assumption \ref{ass:diff} holds.
\end{cor}

\begin{cor}
    Let $p\in(1,\infty)$. Assume that \eqref{ass:accessible} holds. If $m(E)<\infty$, then identity \eqref{eq:HSwithoutSS} can by rewritten in the following form
    \begin{align}
        \intE |f|^p \,dm - \frac{1}{m(E)^{p-1}} \left| \intE f \,dm \right|^p
        =
        \int\limits_0^{\infty} \quad\iintEEd
        F_p(P_tf(x),P_tf(y)) \,J(dx,dy)dt.
    \end{align}
\end{cor}
\begin{proof}
    It is enough to utilize \eqref{eq:HSwithoutSS}, Fact~\ref{fact:Pinfty=bar}
    and notice that
    $\lim_{T\to\infty} \|P_T f\|_p^p = \|\Bar{f}\|_p^p = m(E)^{1-p}\left| \int_E f\,dm\right|^p$.
\end{proof}

\section{Continuity of $P_tf$ for Dirichlet forms on $d$\protect\nobreakdash-sets}
\label{apx:d-set}

In this section, we sketch the proof of condition \ref{ass:continuous} within the context of Example~\ref{ex:d-set}.

When assumption \eqref{ass:dsets} holds, the semigroup $(P_t)_{t\geq0}$ has H\"older continuous kernel $p_t(x,y)$ such that for some constants $c_1,c_2>0$
\begin{align}
\label{neq:ptneq}
    c_1 \min\left\{\frac{1}{t^{d/\alpha}}, \frac{t}{|x-y|^{d+\alpha}}\right\}
    \leq p_t(x,y) \leq
    c_2 \min\left\{\frac{1}{t^{d/\alpha}}, \frac{t}{|x-y|^{d+\alpha}}\right\}
\end{align}
for all $x,y\in E$ and $t\in(0,1]$; see Theorem 1.1. and Theorem 4.14. in \cite{MR2008600}. This means that $P_tf$ ($t>0$) is given by
\begin{align*}
    P_tf(x) = \intE f(y)p_t(x,y)\,m(dy),
    \quad
    x\in E,
\end{align*}
for every $f\in\LpE$, $p\in[1 ,\infty]$.

To prove assumption \ref{ass:continuous} we will use the upper bound from \eqref{neq:ptneq} and H\"older continuity of $p_t$.

Let us first consider $f\in L^1\EAm$ and $t\in(0,1)$. Then there are the constants $c(t),\beta>0$ such that
$
    |p_t(x_1,y)-p_t(x_2,y)|
    \leq
    c(t) |x_1-x_2|^\beta
$
for all $x_1,x_2,y\in E$. By a simple calculation
\begin{align*}
    |P_tf(x_1) - P_tf(x_2)|
    \leq
    c(t) |x_1-x_2|^\beta \|f\|_1,
\end{align*}
hence $P_tf\in C(E)$.

Now, let $f\in L^\infty\EAm$ and $t\in(0,1)$. Let $f_k:=f\indyk_{B(x_0,k)}$ for some arbitrary $x_0\in E$. Let $K\subseteq E$ be any compact set. Then for all $x\in K$, by \eqref{neq:ptneq} and standard calculations, we can show that for sufficiently large $k$
\begin{align*}
    |P_tf(x) - P_tf_k(x)|
    \leq
    c_2\|f\|_\infty
    \int\limits_{\mathclap{E\setminus B(x_0,k)}} \frac{m(dy)}{|y-x_0|^{d+\alpha}}.
\end{align*}
Therefore, $P_tf_k$ convergences uniformly to $P_tf$ on $K$ when $k\to\infty$ for every compact set $K$. Since $f_k\in L^1\EAm$, $P_tf_k\in C(E)$, hence also $P_tf\in C(E)$.

Finally, for $f\in\LpE$ with any $p\in(1,\infty)$ it is enough to use decomposition
$
    f = f\indyk_{\{|f|\geq1\}} + f\indyk_{\{|f|<1\}}
    \in
    L^1\EAm + L^\infty\EAm.
$





\normalsize
\baselineskip=17pt









\printbibliography

\end{document}